\documentclass[11pt]{article}
\usepackage{amssymb,amsmath}
\usepackage{mathrsfs}
\usepackage{mathrsfs}
\usepackage{amsfonts}
\usepackage{mathrsfs}
\usepackage{graphicx}
\usepackage{bbm}
\usepackage{amsmath,amsthm,amssymb,amscd}
\usepackage{latexsym}
\usepackage[numbers,sort&compress]{natbib}
\usepackage[all]{xy}
\usepackage{amssymb}
\usepackage{geometry}
\usepackage{algorithm}
\usepackage{algorithmic}
\newtheorem{lemma}{Lemma}[section]
\newtheorem{theorem}{Theorem}[section]
\newtheorem{definition}{Definition}[section]
\newtheorem{problem}{Problem}[section]

\newtheorem{example}{Example}[section]
\newtheorem{remark}{Remark}[section]
\newtheorem{corollary}{Corollary}[section]

\begin{document}
\date{}
\title{{Smith Form Equivalence for Several Classes of Multivariate Polynomial Matrices}
{\thanks{This research was supported by the National Natural Science Foundation of China under Grant Nos.12271154, 12371507 and 12401662.
$\dagger$Corresponding author: Dongmei Li, Email: dmli@hnust.edu.cn.}}
{\author{\small Zuo Chen$^1$  \quad Jiancheng Guan$^2$  \quad  Dongmei Li$^{2,\dagger}$\\
\small \, \, 1\ School of Computer Science and Engineering, Hunan
University of   Science and\\  \small Technology,
 Xiangtan, Hunan, China, 411201\\
 \small \, \, 2\ School of Mathematics and Statistics, Hunan
University of   Science and\\  \small Technology,
 Xiangtan, Hunan, China, 411201
 }}
}
\maketitle

{\bf Abstract}
\quad  
This paper investigates the equivalence reduction for several classes of multivariate polynomial matrices and their Smith forms, establishing some criteria for such reduction.
In particular, we employ algebra isomorphisms as a key tool to study this equivalence problem. We then leverage the Quillen-Suslin and Lin-Bose theorems to extend these results to non-square and rank-deficient matrices. Moreover, the verification of our criteria is achievable algorithmically via existing Gr\"{o}bner basis methods.

\vspace{0.2cm}

\section{Introduction}
Multidimensional ($nD$) systems, such as iterative learning control systems and $nD$ finite-impulse response filter banks, are widely employed in engineering fields like signal processing and dynamic control. Since these systems are represented by multivariate polynomial matrices, the key properties of $nD$ systems including controllability, stability, and solvability can be directly derived from the algebraic characteristics of the corresponding polynomial matrices,   \cite{1,2,3,5,2025,53,65}.
A critical approach in analyzing a given $nD$ system  is to reduce the corresponding multivariate polynomial matrix to its Smith form. This form contains fewer equations and unknowns than the original system while preserving essential structural information. Symbolic computation techniques, such as Euclidean division and Gr\"{o}bner bases, are robust tools for investigating the equivalence and reduction of multivariate polynomial matrices.  

For univariate polynomial matrices, the equivalence problem to their Smith forms has been fully resolved \cite{33,34}. In fact, every univariate polynomial matrix is equivalent to its Smith form. 
However, this equivalence does not hold universally for multivariate polynomial matrices (where the number of variables \(n \geq 2\)), because multivariate polynomial rings are not principal ideal domains and lack the Euclidean division property.
There is a lack of mature theoretical frameworks for this equivalence problem. Establishing it may be a long-standing open challenge in $nD$ system theory.  
In recent decades, researchers have made significant progress on this challenge for some classes of multivariate polynomial matrices. Frost and Storey were among the first to investigate the equivalence of bivariate polynomial matrices to their Smith forms, and they obtained some preliminary results \cite{32,41}.

The Quillen-Suslin theorem \cite{35, 36}, originating from the famous Serre's conjecture \cite{360}, states that a matrix \( F \in K[z_1,z_2,\ldots,z_n]^{l \times m} \), if the \( l \times l \) minors generate the unit ideal, then \( F \) is equivalent to its Smith form \( (I_l\  0_{l, m-l}) \).
In 2006, Lin et al. \cite{11} leveraged the Quillen-Suslin theorem to prove that a square multivariate polynomial matrix \(F\) with \(\det(F) = z_1 - f(z_2, \ldots, z_n)\) is equivalent to its Smith form $diag\{1,\ldots,1,\det(F)\}$. 
Subsequently, Li et al. \cite{13,12} and Lu et al. \cite{171,172} conducted a series of studies on the equivalence between matrices satisfying \(\det(F) = (z_1 - f(z_2, \ldots, z_n))^q\) (where $q$ is a positive integer) and their special types of Smith forms. Furthermore, Liu and Li \cite{37} established a necessary and sufficient condition for such matrices with \(\det(F) = (z_1 - f(z_2, \ldots, z_n))^q\) to be equivalent to their Smith forms of arbitrary structure, thereby completely solving the equivalence problem for this class of matrices.
In 2023, Zheng et al. \cite{16} derived the executable conditions for reducing \( F \in K[x,y]^{l \times l} \) with \( \det(F) = p^t \) to its Smith form, where \( p \in K[x] \) is irreducible. Subsequently, Guan et al. \cite{49,38} extended the results in \cite{16} to the multivariate case by generalizing Vaserstein's global-local theorem, facilitating the broader application of this class of matrices. Although research progress in the Smith form equivalence of multivariate polynomial matrices has been achieved, a complete solution remains unknown. 



Recent work by Lu et al. \cite{50} has made significant progress by examining the matrix $F \in {K[z_1,z_2,\ldots,z_n]^{l \times l}}$ with $\det(F) \in K[z_1]$. Through innovative localization techniques, they established that such matrices are equivalent to their Smith forms if and only if the reduced minors of all orders generate the unit ideal.
Building upon this research for polynomial matrices with univariate determinants, we extend the investigation to broader classes of matrices and consider the following problems. 
\begin{problem}\label{pro1}
Let $F  \in {K[z_1,z_2,\ldots,z_n]^{l \times l}}$ and $\det(F) =(z_1-f(z_2,\ldots,z_n))h$, where $h\in K[z_n]$. What is the necessary and sufficient condition for the equivalence of $F$ and its Smith form?
\end{problem}
\begin{problem}\label{pro2}
Let $F  \in {K[z_1,z_2,\ldots,z_n]^{l \times l}}$ and $\det(F) =(z_1-f_1(z_2,\ldots,z_n))\cdots (z_{n-1}-f_{n-1}(z_n))h$, where $h\in K[z_n]$. What is the necessary and sufficient condition for the equivalence of $F$ and its Smith form?
\end{problem}

Leveraging  the determinant structure analyzed in the previous problems, we consider linear polynomial factors of the following form and introduce a homomorphic mapping.
Let \( g_i \in K[z_1,\ldots,z_{n-1}] \) be $n-1$ linear polynomials of the form  
$
g_i = a_{i1}z_1 + a_{i2}z_2 + \ldots + a_{i,n-1}z_{n-1} + b_i,
$  
where \( a_{ij}, b_i \in K \) for \( i = 1,\ldots,n-1 \).
We define a homomorphic mapping $\phi$ of $K[z_1,z_2,\ldots,z_n]$ as follows:
\[
\begin{aligned}
\phi :\quad 
z_1 \mapsto g_1,\quad z_2 \mapsto g_2, \quad
\ldots, \quad
z_{n-1}\mapsto g_{n-1},\quad 
z_n \mapsto z_n.\quad
\end{aligned}
\]
Then we use it to investigate a new class of polynomial matrices and consider the following problem.
\begin{problem}\label{pro3}
Let $F  \in {K[z_1,z_2,\ldots,z_n]^{l \times l}}$ and $\det(F) =g_1 g_2 \cdots g_{n-1} h$, where $h\in K[z_n]$. When is the $F$ equivalent to its Smith form?
\end{problem}

The rest of the paper is organized as follows. In Section 2, we review the fundamental concepts regarding polynomial matrix equivalence and introduce several auxiliary lemmas. In Section 3, we address Problems \ref{pro1} and \ref{pro2} and present the corresponding main theoretical results. In Section 4, we employ a ring isomorphism to tackle Problem \ref{pro3} and extend the equivalence results to a new class of matrices. Finally, in Section 5, we provide the conclusions.

\section{Preliminaries}
Let $K[{z_1},{z_2}, \ldots ,{z_n}]$ be the polynomial ring in variables ${z_1},{z_2}, \ldots ,{z_n}$  with coefficients in the field $K$, 
and let $\overline{K}$ be an algebraically closed field containing $K$. 
$K[{z_1},{z_2}, \ldots ,{z_n}]^{l \times m}$ denotes the set of $l \times m$ matrices with entries from $K[{z_1},{z_2}, \ldots ,{z_n}]$.
${0_{s,t}}$ denotes the $s \times t$ zero matrix, and $I_t$ denotes the $t \times t$ identity matrix. We will simply write $F$ to represent  $F({z_1},{z_2}, \ldots ,{z_n})$ when there is no ambiguity. Let $h_1, \ldots , h_l\in K[{z_1},{z_2}, \ldots ,{z_n}]$,  $diag\{h_1, \ldots , h_l\}$ denotes the diagonal matrix with the $i$-th diagonal entry being $h_i$ for $i = 1, \ldots , l$.

We begin by presenting some necessary definitions and preliminary results.
\begin{definition}
Let $F \in {K[{z_1},{z_2}, \ldots ,{z_n}]^{l \times l}}$. Then $F$ is said to be unimodular if the determinant of $F$ is a unit in $K[{z_1},{z_2}, \ldots ,{z_n}]$.
\end{definition}
\begin{definition}Let $F\in K[{z_1},{z_2}, \ldots ,{z_n}]^{l\times m}$ ($l\leq m$). The Smith  form of $F$ is defined as
 $$S=(diag\{\Phi_i(F)\}\ \ 0_{l, m-l}), $$
where
\begin{equation*}
\Phi_i(F)=\left\{
\begin{aligned}
&d_{i}(F)/d_{i-1}(F),     &&1\leq i\leq r; \\
&0,     &&r<i\leq l,
\end{aligned}
\right.
\end{equation*}
$r$ is the rank of $F$, $d_0(F)\equiv 1$, $d_i(F)$ is the greatest common divisor of the $i\times i$ minors of $F$ and $\Phi_i$ satisfies the divisibility property $\Phi_1\mid\Phi_2\mid\cdots\mid\Phi_r.$
\end{definition}

\begin{definition}(\cite{111})
Let $F \in {K[{z_1},{z_2}, \ldots ,{z_n}]^{l \times m}}$ with rank $r$, where $1 \le r \le min\{l, m\}$. For any given integer $k$ with $1 \le k \le r$, let ${u_{\rm{1}}}, \ldots ,{u_\beta }$ be all the $k \times k$ minors of $F$ and  $d_k(F)$ be the greatest common divisor $(g.c.d.)$ of ${u_{\rm{1}}}, \ldots ,{u_\beta }$. Extracting $d_k(F)$ from ${u_{\rm{1}}}, \ldots ,{u_\beta }$ yields
\[{u_j} = {d_k}(F) \cdot {v_j}, \ \  j = 1, \ldots ,\beta .\]
Then, $v_1, \ldots , v_\beta$ are called the $k \times k$ reduced minors of $F$. In addition, we use $J_k(F)$ to denote the ideal generated by $v_1, \ldots , v_\beta$.
\end{definition}

\begin{definition}(\cite{22})
Let $F \in {K[{z_1},{z_2}, \ldots ,{z_n}]^{l \times m}}$ be of full row (column) rank. Then $F$ is said to be zero left prime (zero right prime) if the $l \times l$\ ($m \times m$) minors of $F$ have no common zeros in $\overline{K}^n$, i.e., generate unit ideal $K[{z_1},{z_2}, \ldots ,{z_n}]$. If $F \in {K[{z_1},{z_2}, \ldots ,{z_n}]^{l \times m}}$ is zero left prime (zero right prime), then $F$ is called simply to be ZLP (ZRP).
\end{definition}

\begin{definition}
Let $F_1,F_2\in K[{z_1},{z_2}, \ldots ,{z_n}]^{l\times m}$. Then $F_1$ and $F_2$ are said to be equivalent if there exist two unimodular matrices $U\in K[{z_1},{z_2}, \ldots ,{z_n}]^{l\times l}$ and $V\in K[{z_1},{z_2}, \ldots ,{z_n}]^{m\times m}$ satisfying $F_1=UF_2V$. This relation is denoted by $F_1 \sim F_2$.
\end{definition}

\begin{lemma}(Quillen-Suslin Theorem, \cite{35,36})\label{le1}
Let $F \in {K[{z_1},{z_2}, \ldots ,{z_n}]^{l \times m}}(l\le m)$ be a ZLP matrix. Then there exists a unimodular matrix $U \in {K[{z_1},{z_2}, \ldots ,{z_n}]^{m \times m}}$ such that $F \cdot U = \left( {\begin{array}{*{20}{c}}
{{I_l}}&{{0_{l,m - l}}}
\end{array}} \right).$
\end{lemma}

\begin{lemma}\label{le2}(\cite{37})
Let $A, B \in {K[{z_1},{z_2}, \ldots ,{z_n}]^{l \times m}}$. If $A\sim B$, then $d_k(A) = d_k(B)$ and $J_k(A) =J_k(B) $, where $k = 1,2,\ldots ,\min \{ m,l\} $.
\end{lemma}

\begin{lemma}\label{le4}(\cite{37})
Let $F,M,N \in {K[{z_1},{z_2}, \ldots ,{z_n}]^{l \times l}}$, $F=M\cdot N$. For some $k(1\le k\le l)$, if ${d_k}(M) = {d_k}(F)$, ${J_k}(F) = K[{z_1},{z_2}, \ldots ,{z_n}]$, then ${J_k}(M) = K[{z_1},{z_2}, \ldots ,{z_n}]$, ${d_k}(N) = 1$, ${J_k}(N) = K[{z_1},{z_2}, \ldots ,{z_n}]$.
\end{lemma}


\begin{lemma}\label{le5}(\cite{38})
Let $F \in {K[z_1,{z_2}, \ldots ,{z_n}]^{l \times l}}$ and $p \in K[{z_1}]$ be an irreducible polynomial. If ${J_k}(F) = K[z_1,{z_2}, \ldots ,{z_n}]$ and $p \mid d_{k+1}(F)$, then there exists a unimodular matrix $U \in {K[z_1,{z_2}, \ldots ,{z_n}]^{l \times l}}$ such that
\[U \cdot F = \left( {\begin{array}{*{20}{c}}
{{I_k}}&{}\\
{}&{p{I_{l - k}}}
\end{array}} \right) \cdot G,\]
where $G\in {K[z_1,{z_2}, \ldots ,{z_n}]^{l \times l}}$.
\end{lemma}

\begin{lemma}\label{le61}(\cite{12})
Let $F \in {K[{z_1},{z_2}, \ldots ,{z_n}]^{l \times m}}$ with rank $r$. If ${J_r}(F) = K[z_1,{z_2}, \ldots ,{z_n}]$, then there exists a ZLP matrix $W \in {K[z_1,{z_2}, \ldots ,{z_n}]^{(l - r) \times l}}$ such that $W \cdot F= {0_{l - r, m}}.$
\end{lemma}

\begin{lemma}\label{le6}
Let $F \in {K[{z_1},{z_2}, \ldots ,{z_n}]^{l \times l}}$ and $(z_1-f(z_2,\ldots,z_n))\mid\det(F)$.
If ${d_k}(F) = 1$, ${J_k}(F) = K[{z_1},{z_2}, \ldots ,{z_n}]$ and $(z_1-f(z_2,\ldots,z_n))\mid{d_{k + 1}}(F)$, then there exists a unimodular matrix $U \in {K[{z_2}, \ldots ,{z_n}]^{l \times l}}$ such that
\[{U} \cdot F = \left( {\begin{array}{*{20}{c}}
{{I_k}}&{}\\
{}&{{(z_1-f(z_2,\ldots,z_n))}{I_{l - k}}}
\end{array}} \right) \cdot {G}\]
where ${G} \in {K[{z_1},{z_2}, \ldots ,{z_n}]^{l \times l}}$.
\end{lemma}
\begin{proof}
Suppose that the $k\times k$ minors of $F$ are ${a_1},{a_2}, \ldots ,{a_\beta }$ and $F_1 = F({f},{z_2}, $ $\ldots ,{z_n})$, the $k\times k$ minors of $F_1$ are ${b_1},{b_2}, \ldots ,{b_\beta }$. It is clear that $({f},{z_2}, \ldots ,{z_n})$ is a zero of $\det (F)$ for every $({z_2}, \ldots ,{z_n}) \in \overline{K}^{n-1}$ and $(z_1-f({z_2}, $ $\ldots ,{z_n}))\mid {d_{k+ 1}}(F)$. Therefore, $rank(F_1) \le k$.
Assume that there exists $({z_{20}}, \ldots ,{z_{n0}}) \in \overline{K}^{n-1}$ such that
\[{b_i}({z_{20}}, \ldots ,{z_{n0}}) = 0,\ i = 1,2, \ldots ,\beta.\]
Let ${z_{10}} = {f}({z_{20}}, \ldots ,{z_{n0}})$. Then we have
\[{a_i}({z_{10}},{z_{20}}, \ldots ,{z_{n0}}) = 0,\ i = 1,2, \ldots ,\beta. \]
Since ${d_k}(F) = 1$, ${J_k}(F) = K[z_1,{z_2}, \ldots ,{z_n}]$, we can obtain that all the $k\times k$ minors of $F$ have no common zeros in $\overline{K}^n$,  a contradiction. Thus, the $k\times k$ minors of $F_1$ generate unit ideal $K[{z_2}, \ldots ,{z_n}]$ and $rank(F_1) \ge k$. It follows that $rank(F_1) = k$. According to Lemma \ref{le61}, there exists a ZLP matrix $W \in {K[{z_2}, \ldots ,{z_n}]^{(l - k) \times l}}$ such that $W \cdot F_1= {0_{l - k, l}}.$
By Lemma \ref{le1}, a unimodular matrix $U \in {K[{z_2}, \ldots ,{z_n}]^{l \times l}}$ can be constructed and $W$ is its last $l-k$ rows. Therefore, the last $l-k$ rows of matrix $U \cdot F$ have a common divisor $z_1-f(z_2,\ldots,z_n)$, i.e.,
\[{U} \cdot F = \left( {\begin{array}{*{20}{c}}
{{I_k}}&{}\\
{}&{{(z_1-f(z_2,\ldots,z_n))}{I_{l - k}}}
\end{array}} \right) \cdot {G},\]
where ${G} \in {K[{z_1},{z_2}, \ldots ,{z_n}]^{l \times l}}$. 
\end{proof}





\begin{lemma}\label{le9}(\cite{50})
Let $F,F_1,F_2 \in {K[{z_1},{z_2}, \ldots ,{z_n}]^{l \times l}}$ satisfy $F = F_1F_2$, $gcd(det(F_1), det(F_2)) = 1$ and $J_i(F) = K[{z_1},{z_2}, \ldots ,{z_n}]$ for $i = 1, \ldots , l$. Then $d_i(F) = d_i(F_1) \cdot d_i(F_2)$ and $J_i(F_1) = J_i(F_2) = K[{z_1},{z_2}, \ldots ,{z_n}]$, where $i = 1, \ldots , l$.
\end{lemma}

\begin{lemma}\label{le8}(\cite{50})
Let $F \in {K[{z_1},{z_2}, \ldots ,{z_n}]^{l \times m}}$ with rank $r$ and $d_r(F)\in K[{z_1}]$. Then $F$ is equivalent to its Smith form if and only if $J_k(F)=K[{z_1},{z_2}, \ldots ,{z_n}]$ for $k = 1, \ldots, r$.
\end{lemma}

%
%

\section{Equivalence to Smith form  for two classes of matrices}
In this section, we solve Problems \ref{pro1} and \ref{pro2}.
Let matrix $F \in K[z_1, z_2, \ldots, z_n]^{l \times l}$. We first consider the case where $\det(F) = (z_1 - f(z_2, \ldots, z_n))h$ with $h \in K[z_n]$, and seek the conditions under which $F$ is equivalent to its Smith form. Subsequently, we generalize this result to matrices with $\det(F) = (z_1 - f_1(z_2, \ldots, z_n))\cdots(z_{n-1} - f_{n-1}(z_n))h$. Finally, we present an example to illustrate the main results.


\begin{theorem}\label{th1}
Let
$$
B =diag\{ h_1,h_2,\ldots,h_l\}\cdot V \cdot diag\{1,\ldots,1,z_1-f(z_2,\ldots,z_n)\},
$$
where $h_1 ,h_2,\ldots, h_l\in K[z_n]$ satisfy $h_1\mid h_2 \mid \cdots \mid h_l$ and $V \in {K[z_2,\ldots,z_n]^{l \times l}}$ is a unimodular matrix. If $d_k(B) = h_1\cdots h_k$ and $J_k(B)=K[z_1,z_2,\ldots,z_n]$, $k = 1, \ldots, l-1$, then
$$
B\sim diag\{ h_1,h_2,\ldots,(z_1-f(z_2,\ldots,z_n))h_l\}.
$$
\end{theorem}
\begin{proof}
Let $V = (v_{ij})_{l\times l} =  \left( {\begin{array}{*{20}{c}}
{{V_1}}&{{V_2}}\\
{{V_3}}&{{V_4}}
\end{array}} \right)$, where $V_1\in {K[z_2,\ldots,z_n]^{(l-1) \times (l-1)}}$, $V_2\in {K[z_2,\ldots,z_n]^{(l-1) \times 1}}$, $V_3\in {K[z_2,\ldots,z_n]^{1 \times (l-1)}}$ and $V_4=v_{l,l}\in {K[z_2,\ldots,z_n]^{1 \times 1}}$. Then we have
\begin{align}
B &= diag\{ h_1,h_2,\ldots,h_l\}  \cdot V \cdot diag\{1,\ldots,1,z_1-f(z_2,\ldots,z_n)\}
\nonumber  \\&= \left( {\begin{array}{*{20}{c}}
h_1v_{1,1}&{h_1v_{1,2}}&{\cdots}&(z_1-f(z_2,\ldots,z_n))h_1{v_{1,l}}\\
{h_2v_{2,1}}& h_2v_{2,2} &{\cdots}&(z_1-f(z_2,\ldots,z_n))h_2{v_{2,l}}\\
{\vdots}&{\vdots}& \ddots &{\vdots}\\
{h_lv_{l,1}}&{h_lv_{l,2}}&{\cdots}&{(z_1-f(z_2,\ldots,z_n))h_lv_{l,l}}
\end{array}} \right).\nonumber
\end{align}
Let $B_1$ denote the submatrix formed by the first $l-1$ columns of $B$:
\[B_1 = \left( {\begin{array}{*{20}{c}}
h_1v_{1,1}&{h_1v_{1,2}}&{\cdots}&h_1{v_{1,l-1}}\\
{h_2v_{2,1}}& h_2v_{2,2} &{\cdots}&h_2{v_{2,l-1}}\\
{\vdots}&{\vdots}& \ddots &{\vdots}\\
{h_lv_{l,1}}&{h_lv_{l,2}}&{\cdots}&{h_lv_{l,l-1}}
\end{array}} \right).\]
We begin by proving that $d_i(B_1) =  d_i(B)$ for $i=  1, \ldots,l-1$. Given that $d_i(B) = h_1\cdots h_i$ for $i = 1, \ldots, l-1$, it follows that all the $i\times i$ reduced minors of $B$ are:
\begin{equation}
a_{i1},\ldots, a_{ik_0},h'a_{ik_1},\ldots, h''a_{i\beta}, b_{i1},\ldots, b_{i\xi},
\end{equation}
where $a_{i1},\ldots, a_{ik_0},h'a_{ik_1},\ldots, h''a_{i\beta}$ are obtained by dividing the $i\times i$ minors of $B_1$ by $d_i(B)$, $b_{i1},\ldots, b_{i\xi}$ are the remaining reduced minors of $B$.  
The fact that $d_i(B)\mid d_i(B_1)$ implies that there exist some polynomials $q_i\in K[{z_2}, \ldots ,{z_n}]$ such that $d_i(B_1) = q_i\cdot d_i(B)$.
We claim that $q_i$ is a nonzero constant for each $i=1,\ldots,l-1$. Otherwise,  we can conclude that $a_{i1},\ldots, a_{ik_0},h'a_{ik_1},\ldots, h''a_{i\beta}$ have a common factor $q_i$, which in turn implies that these minors have a common zero $(z_{20},\ldots,z_{n0})\in \overline{K}^{n-1}$. In addition, from the fact that $(z_1-f(z_2,\ldots,z_n))\mid b_{ij}$ for each $j = 1, \ldots, \xi$, the $b_{ij}$ therefore have a common zero $\alpha_0 = (f(z_{20},\ldots,z_{n0}),z_{20},\ldots,z_{n0}) \in \overline{K}^{n}$.
It follows that all the $i\times i$ reduced minors of $B$ have a common zero $\alpha_0$, which contradicts $J_i(B)=K[{z_1},{z_2}, \ldots ,{z_n}]$. Hence, each $q_i$ is a non-zero constant. Without loss of generality, we may assume that $q_i=1$ for all $i$, so that $d_i(B_1) =  d_i(B)=h_1\cdots h_i$, $i=  1, \ldots,l-1$. 

Next, we prove that $J_i(B_1)=K[{z_1},{z_2}, \ldots ,{z_n}]$ for $i=  1, \ldots,l-1$. Since $B_1 \in {K[{z_2}, \ldots ,{z_n}]^{l \times l}}$, we observe that if all the $i\times i$ reduced minors of $B_1$ have a common zero; then, given that $(z_1-f(z_2,\ldots,z_n))\mid b_{ij}$ with $j = 1, \ldots, \xi$, all the $i\times i$ reduced minors of $B$ must also have a common zero. This yields a contradiction.

According to Lemma \ref{le8}, there are two unimodular matrices $U_1\in {K[{z_1},{z_2}, \ldots ,{z_n}]^{l \times l}}$ and $U_2\in {K[{z_1},{z_2}, \ldots ,{z_n}]^{(l-1) \times (l-1)}}$ such that
\[U_1B_1U_2 = 
\left( {\begin{array}{*{20}{c}}
h_1 & \cdots & 0\\
\vdots & \ddots & \vdots\\
0 & \cdots & h_{l-1}\\
0 & \cdots & 0
\end{array}} \right).\]
Let
$U'_2 = \left( {\begin{array}{*{20}{c}}
{{U_2}}&{}\\
{}&{{1}}
\end{array}} \right)$. 
Then
\begin{align}
U_1BU'_2 = \left( {\begin{array}{*{20}{c}}
h_1&{0}&{\cdots}&{0}&{(z_1-f(z_2,\ldots,z_n))v'_{1,l}}\\
{0}& h_2 &{\cdots}&{0}&{(z_1-f(z_2,\ldots,z_n))v'_{2,l}}\\
{\vdots}&{\vdots}&\ddots&{\vdots}&{\vdots}\\
{0}&{0}&{\cdots}&h_{l-1}&(z_1-f(z_2,\ldots,z_n))v'_{l-1,l} \\
{0}&{0}&{\cdots}& 0 & (z_1-f(z_2,\ldots,z_n))v'_{l,l} 
\end{array}} \right),\nonumber
\end{align}
where
$(v'_{1,l},v'_{2,l},\ldots,v'_{l,l})^T = U_1\cdot diag\{ h_1,\ldots,h_l\}\cdot\left( {\begin{array}{*{20}{c}}
{{V_2}}\\
{{V_4}}
\end{array}} \right).$
Since $h_1 \mid v{'_{1,l}},h_2 \mid v{'_{2,l}},\ldots, h_{l-1} \mid v{'_{l-1,l}}$, we can derive the following equivalence relation after a series of elementary column operations:
\[B \sim C = \left( {\begin{array}{*{20}{c}}
h_1&{0}&{\cdots}&{0}&{0}\\
{0}& h_2 &{\cdots}&{0}&{0}\\
{\vdots}&{\vdots}&\ddots&{\vdots}&{\vdots}\\
{0}&{0}&{\cdots}&h_{l-1}&0 \\
{0}&{0}&{\cdots}& 0 & (z_1-f(z_2,\ldots,z_n))v'_{l,l} 
\end{array}} \right).\]
Since $\det(C) = \det(B) = (z_1-f(z_2,\ldots,z_n))h_1\cdots h_{l-1}h_l$, we have $v'_{l,l} = h_l$.
\end{proof}

The following corollary follows immediately from Theorem \ref{th1}.
\begin{corollary}\label{co3.1}
Let
$
B =diag\{ h_1,h_2,\ldots,h_l\}\cdot V \cdot diag\{1,\ldots,1,z_1\},
$
where $h_1 ,h_2,\ldots, h_l\in K[z_n]$ satisfy $h_1\mid h_2 \mid \cdots \mid h_l$ and $V \in {K[z_2,\ldots,z_n]^{l \times l}}$ is a unimodular matrix. If $d_k(B) = h_1\cdots h_k$ and $J_k(B)=K[z_1,z_2,\ldots,z_n]$, $k = 1, \ldots, l-1$, then
$
B\sim diag\{ h_1,h_2,\ldots,z_1h_l\}.
$
\end{corollary}


We now prove the following theorem, which provides a positive answer to Problem \ref{pro1}.

\begin{theorem}\label{th6}
Let $F \in {K[z_1,z_2,\ldots,z_n]^{l \times l}}$ and $\det(F)=(z_1-f(z_2,\ldots,z_n))h$, where $h\in K[z_n]$. Then $F$ is equivalent to its Smith form if and only if $J_k(F) = K[{z_1},{z_2}, \ldots ,{z_n}]$ for $k=1,\ldots,l$.
\end{theorem}
\begin{proof}
Without loss of generality, assume that the Smith form of $F$ is
$$
S = diag\{h_1,h_2, \ldots,  (z_1-f(z_2,\ldots,z_n))h_l\},
$$
where $h_1,h_2,\ldots,h_l\in K[z_n]$ satisfy $h_1\mid h_2\mid \cdots \mid h_l$ and $h_1\cdots h_l=h$.

Necessity: If $F \sim S$, by Lemma \ref{le2}, $J_k(F) =J_k(S) = K[{z_1},{z_2}, \ldots ,{z_n}]$, where $k = 1, \ldots , l$.

Sufficiency: Let $\theta$ be the algebra automorphism induced by $\theta(z_1)=z_1-f(z_2,\ldots,z_n)$ and $\theta(z_i)=z_i$ for $i=2,\ldots, n$. To complete the proof, it suffices to prove that the matrix $F$ with $\det(F)=z_1h$ is equivalent to its Smith form $diag\{h_1,h_2, \ldots,  z_1h_l\}$.
Since $J_{l-1}(F) = K[{z_1},{z_2}, \ldots ,{z_n}]$ and $z_1 \mid \det(F)$, by Lemma \ref{le5}, we obtain that there exists a unimodular matrix $U \in K[{z_1},{z_2}, \ldots ,{z_n}]^{l\times l}$ such that $U\cdot F = diag\{1, \ldots,1, z_1\} \cdot {F_1}$, where $F_1 \in {K[{z_1},{z_2}, \ldots ,{z_n}]^{l \times l}}$. It is clear that $\det(F_1) = h$ and $\gcd(z_1, h) = 1$. By Lemma \ref{le9}, we have $d_k(F_1)= h_1\cdots h_k$ and $J_k(F_1) = K[{z_1},{z_2}, \ldots ,{z_n}]$ for $k=1,\ldots,l-1$. Further, by Lemma \ref{le8}, $F_1$ is equivalent to its Smith form, i.e., there exist two unimodular matrices $U',V' \in K[{z_1},{z_2}, \ldots ,{z_n}]^{l\times l}$ such that 
$$F_1 = U'\cdot diag\{h_1,h_2, \ldots,  h_l\}\cdot V'.$$
Hence, we can rewrite $F$ as follows: 
$$F = U^{-1}\cdot diag\{1, \ldots,1, z_1\}\cdot U'\cdot diag\{h_1,h_2, \ldots,  h_l\}\cdot V'.
$$
Next, we consider the transpose of matrix $F$. From the equation above, we have 
$$F^T\sim diag\{h_1,h_2,\ldots,h_2\}\cdot (U')^T \cdot diag\{1, \ldots,1, z_1\}.$$ Let $A \triangleq (U')^T\cdot diag\{1, \ldots,1, z_1\}$. Then $F^T\sim diag\{h_1,h_2,\ldots,h_2\}\cdot A$. From Lemma \ref{le9} again, we have $J_{l-1}(A) = K[{z_1},{z_2}, \ldots ,{z_n}]$ and $z_1 \mid \det(A)$. By Lemma \ref{le4}, $d_{l-1}(A) =1$. Therefore, by Lemma \ref{le6}, there exists a unimodular matrices $W \in K[{z_2}, \ldots ,{z_n}]^{l\times l}$ such that $W\cdot A = diag\{1, \ldots,1, z_1\} \cdot G$, where $G \in {K[{z_1},{z_2}, \ldots ,{z_n}]^{l \times l}}$. 
By using Corollary \ref{co3.1}, we have $$diag\{ h_1,h_2,\ldots,h_l\}\cdot W^{-1} \cdot diag\{1,\ldots,1,z_1\} \sim diag\{ h_1,h_2,\ldots,z_1h_l\}.$$
Therefore, we have $F^T \sim diag\{h_1,h_2,\ldots, z_1h_l\}$. It follows that $F \sim diag\{h_1,h_2,\ldots, z_1h_l\}$. 
\end{proof}

Subsequently, we extend the above result to the case of non-square and rank-deficient matrices with the help of the Lin-Bose theorem and Quillen-Suslin theorem.

\begin{lemma}(Lin-Bose Theorem, \cite{20})\label{le00}
Let $F \in {K[{z_1},{z_2}, \ldots ,{z_n}]^{l \times m}}(l<m)$ with rank $r$, where $1 \le r \le l$. If $J_r(F) = K[{z_1},{z_2}, \ldots ,{z_n}]$, then there exist $G_1 \in K[{z_1},{z_2}, \ldots ,{z_n}]^{l\times r}$ and $F_1 \in K[{z_1},{z_2}, \ldots ,{z_n}]^{r\times m}$ such that $F = G_1F_1$, where $d_r(G_1) = d_r(F)$ and $F_1$ is a ZLP matrix.
\end{lemma}

\begin{theorem}\label{th1.2}
Let $F \in {K[{z_1},{z_2}, \ldots ,{z_n}]^{l \times m}}(l<m)$ with rank $r$ and  $d_r(F) = (z_1-f(z_2,\ldots,z_n))h$, where $h\in K[z_n]$. Then $F$ is equivalent to its Smith form if and only if $J_k(F) = K[{z_1},{z_2}, \ldots ,{z_n}]$ for $k = 1, \ldots , r$.
\end{theorem}
\begin{proof}
The necessity follows immediately from Lemma \ref{le2}. In the following we prove the sufficiency.
According to  Lemma \ref{le00}, there exist two matrices $G \in {K[{z_1},{z_2}, \ldots ,{z_n}]^{l \times r}}$ and $F_1 \in {K[{z_1},{z_2}, \ldots ,{z_n}]^{r \times m}}$ such that $F=GF_1$, where $F_1$ is a ZLP matrix. By using Lemma \ref{le1}, there exists a unimodular matrix $U_1 \in {K[{z_1},{z_2}, \ldots ,{z_n}]^{m \times m}}$ such that $F_1 \cdot U_1 = \left( {\begin{array}{*{20}{c}}
{{I_r}}&{{0_{r,m - r}}}
\end{array}} \right)$. Therefore, we have
\begin{equation}\label{eq2}
FU_1 = GF_1U_1 = G\left( {\begin{array}{*{20}{c}}
{{I_r}}&{{0_{r,m - r}}}
\end{array}} \right) = \left( {\begin{array}{*{20}{c}}
{{G}}&{{0_{r,m - r}}}
\end{array}} \right).
\end{equation}
From equation (\ref{eq2}), it is straightforward to show that $F$ is equivalent to $\left( {\begin{array}{*{20}{c}}
{{G}}&{{0_{r,m - r}}}
\end{array}} \right)$. Then by using Lemma \ref{le2}, we have $J_r(G) = J_r(F) = K[{z_1},{z_2}, \ldots ,{z_n}]$. Applying Lemma \ref{le00} again, we have that there exist two matrices $G' \in {K[{z_1},{z_2}, \ldots ,{z_n}]^{l \times r}}$ and $F'_1 \in {K[{z_1},{z_2}, \ldots ,{z_n}]^{r \times r}}$ such that $G=G'F'_1$, where $G'$ is a ZRP matrix. By Lemma \ref{le1}, there exists a unimodular matrix $U_2 \in {K[{z_1},{z_2}, \ldots ,{z_n}]^{l \times l}}$ such that $U_2  G' = \left( {\begin{array}{*{20}{c}}
{{I_r}}\\
{{0_{l-r, r}}}
\end{array}} \right)$.  Then pre-multiplying equation (\ref{eq2}) by matrix $U_2$, we have
\begin{align}\label{eq3}
U_2FU_1 = U_2GF_1U_1= U_2G'F'_1F_1U_1 = \left( {\begin{array}{*{20}{c}}
{{I_r}}\\
{{0_{l-r, r}}}
\end{array}}\right) F'_1 \left( {\begin{array}{*{20}{c}}
{{I_r}}&{{0_{r,m - r}}}
\end{array}} \right)= \left( {\begin{array}{*{20}{c}}
{{F'_1}}&{{0_{r,m - r}}}\\
{0_{l-r, r}}&0_{{l-r,m - r}}
\end{array}} \right).\nonumber
\end{align}
Since $U_1,U_2$ are unimodular matrices, we have
$$F\sim \left( {\begin{array}{*{20}{c}}
{{F'_1}}&{{0_{r,m - r}}}\\
{0_{l-r, r}}&0_{{l-r,m - r}}
\end{array}} \right).$$
Furthermore, by using Lemma \ref{le2}, $d_r(F)=d_r(F'_1)=\det(F'_1)=(z_1-f(z_2,\ldots,z_n))h$ and $J_k(F)=J_k(F'_1) = K[{z_1},{z_2}, \ldots ,{z_n}]$, where $k=1,\ldots ,r$. Assume that the Smith form of $F'_1$ is $D$, by using Theorem \ref{th6}, we have $F'_1$ is equivalent to $D$. Therefore, $F$ is equivalent to its Smith form $\left( {\begin{array}{*{20}{c}}
{{D}}&{{0_{r,m - r}}}\\
{0_{l-r, r}}&0_{{l-r,m - r}}
\end{array}} \right)$.
\end{proof}

\begin{corollary}\label{co1}
Let $F \in {K[{z_1},{z_2}, \ldots ,{z_n}]^{l \times m}}(l<m)$ with rank $r$ and  $d_r(F) = (z_1-f(z_2,\ldots,z_n))h$, where $h\in K[z_j]$ for some fixed $j\in \{2,\ldots,n\}$. Then $F$ is equivalent to its Smith form if and only if $J_k(F) = K[{z_1},{z_2}, \ldots ,{z_n}]$ for $k = 1, \ldots , r$.
\end{corollary}
\begin{proof}
The result follows by virtue of Theorem \ref{th1.2}.
\end{proof}

 Let $K[{z^c_i}] = K[{z_1},\ldots,{z_{i-1}},z_{i+1}, \ldots ,{z_n}]$ be the polynomial ring that is independent of variable $z_i$. In the following, we give a generalization of Theorem \ref{th1}.

\begin{theorem}\label{th4}
Let
\begin{small}
$$
B =\left( {\begin{array}{*{20}{c}}
h_1&{}&{}&{}\\
{}& h_2 &{}&{}\\
{}&{}& \ddots &{}\\
{}&{}&{}&{h_l}(z_{n-1}-f_{n-1}(z_n))\cdots (z_j-f_{j}(z_{j+1},\ldots,z_n))
\end{array}} \right)\cdot  V\cdot   \left( {\begin{array}{*{20}{c}}
{I_{l-1}}&{}\\
{}&{z_{j-1}-f_{j-1}(z_j,z_{j+1},\ldots,z_n)}
\end{array}} \right),
$$
\end{small}
where $h_1 ,h_2,\ldots, h_l\in K[z_n]$ satisfy $h_1\mid h_2 \mid \cdots \mid h_l$,  $j\in \{2,\ldots,n-1\}$ and $V \in {K[z^c_{j-1}]^{l \times l}}$ is a unimodular matrix. If $d_k(B) = h_1\cdots h_k$ and $J_k(B)=K[z_1,z_2,\ldots,z_n]$, $k = 1, \ldots, l-1$, then
$$
B\sim \left( {\begin{array}{*{20}{c}}
h_1&{}&{}&{}\\
{}& h_2 &{}&{}\\
{}&{}& \ddots &{}\\
{}&{}&{}&{h_l(z_{n-1}-f_{n-1}(z_n))\cdots (z_{j-1}-f_{j-1}(z_j,z_{j+1},\ldots,z_n))}
\end{array}} \right).
$$
\end{theorem}
\begin{proof}
Let $\alpha = (z_{n-1}-f_{n-1}(z_n))\cdots (z_j-f_{j}(z_{j+1},\ldots,z_n))$, $\beta = z_{j-1}-f_{j-1}(z_j,z_{j+1},\ldots,z_n)$ and 
$V = (v_{ij})_{l\times l} =  \left( {\begin{array}{*{20}{c}}
{{V_1}}&{{V_2}}\\
{{V_3}}&{{V_4}}
\end{array}} \right),$
where $V_1\in {K[z^c_{j-1}]^{(l-1) \times (l-1)}}$, $V_2\in {K[z^c_{j-1}]^{(l-1) \times 1}}$, $V_3\in {K[z^c_{j-1}]^{1 \times (l-1)}}$ and $V_4=v_{l,l}\in {K[z^c_{j-1}]^{1 \times 1}}$. Then we have
\begin{align}
B = \left( {\begin{array}{*{20}{c}}
h_1v_{1,1}&{h_1v_{1,2}}&{\cdots}& {h_1v_{1,l-1}} &\beta h_1{v_{1,l}}\\
{h_2v_{2,1}}& h_2v_{2,2} &{\cdots}&{h_2v_{2,l-1}} &\beta h_2{v_{2,l}}\\
{\vdots}&{\vdots}& \ddots &{\vdots}&{\vdots}\\
{\alpha h_lv_{l,1}}&{\alpha h_lv_{l,2}}&{\cdots}&{\alpha h_lv_{l,l-1}} &{\alpha \beta h_lv_{l,l}}
\end{array}} \right).\nonumber
\end{align}
In addition, let $B_1$ denote the submatrix formed by the first $l-1$ columns of $B$:
\[B_1 = \left( {\begin{array}{*{20}{c}}
h_1v_{1,1}&{h_1v_{1,2}}&{\cdots}&h_1{v_{1,l-1}}\\
{h_2v_{2,1}}& h_2v_{2,2} &{\cdots}&h_2{v_{2,l-1}}\\
{\vdots}&{\vdots}& \ddots &{\vdots}\\
{\alpha h_lv_{l,1}}&{\alpha h_lv_{l,2}}&{\cdots}&{\alpha h_lv_{l,l-1}}
\end{array}} \right).\]
Employing a method analogous to the proof of Theorem \ref{th1}, we can also demonstrate that $d_i(B_1) =  d_i(B) = h_1\cdots h_i$ and $J_i(B_1) =  J_i(B)=K[{z_1},{z_2}, \ldots ,{z_n}]$ for $i=  1, \ldots,l-1$.

According to Lemma \ref{le8}, there are two unimodular matrices $U_1\in {K[{z_1},{z_2}, \ldots ,{z_n}]^{l \times l}}$ and $U_2\in {K[{z_1},{z_2}, \ldots ,{z_n}]^{(l-1) \times (l-1)}}$ such that
\[U_1B_1U_2 = 
\left( {\begin{array}{*{20}{c}}
h_1 & \cdots & 0\\
\vdots & \ddots & \vdots\\
0 & \cdots & h_{l-1}\\
0 & \cdots & 0
\end{array}} \right).\]
Let
$U'_2 = \left( {\begin{array}{*{20}{c}}
{{U_2}}&{}\\
{}&{{1}}
\end{array}} \right)$. 
Then
\begin{align}
U_1BU'_2 = \left( {\begin{array}{*{20}{c}}
h_1&{0}&{\cdots}&{0}&{\beta v'_{1,l}}\\
{0}& h_2 &{\cdots}&{0}&{\beta v'_{2,l}}\\
{\vdots}&{\vdots}&\ddots&{\vdots}&{\vdots}\\
{0}&{0}&{\cdots}&h_{l-1}&\beta v'_{l-1,l} \\
{0}&{0}&{\cdots}& 0 &\beta v'_{l,l} 
\end{array}} \right),\nonumber
\end{align}
where
$(v'_{1,l},v'_{2,l},\ldots,v'_{l,l})^T = U_1\cdot diag\{ h_1,\ldots,h_{l-1},\alpha h_l\}\cdot\left( {\begin{array}{*{20}{c}}
{{V_2}}\\
{{V_4}}
\end{array}} \right).$
Since $h_1 \mid v{'_{1,l}},h_2 \mid v{'_{2,l}},\ldots, h_{l-1} \mid v{'_{l-1,l}}$, we can derive the following equivalence relation after a series of elementary column operations:
\[B \sim C = \left( {\begin{array}{*{20}{c}}
h_1&{0}&{\cdots}&{0}&{0}\\
{0}& h_2 &{\cdots}&{0}&{0}\\
{\vdots}&{\vdots}&\ddots&{\vdots}&{\vdots}\\
{0}&{0}&{\cdots}&h_{l-1}&0 \\
{0}&{0}&{\cdots}& 0 & \beta v'_{l,l} 
\end{array}} \right).\]
Since $\det(C) = \det(B) =\alpha \beta h_1\cdots h_{l-1}h_l$, we have $v'_{l,l} = \alpha h_l$.
Therefore, it follows that
$$
B\sim \left( {\begin{array}{*{20}{c}}
h_1&{}&{}&{}\\
{}& h_2 &{}&{}\\
{}&{}& \ddots &{}\\
{}&{}&{}&{h_l(z_{n-1}-f_{n-1}(z_n))\cdots (z_{j-1}-f_{j-1}(z_j,z_{j+1},\ldots,z_n))}
\end{array}} \right).
$$
\end{proof}

Using the results above, we now consider the equivalence of another class of multivariate polynomial matrices to their Smith forms.
\begin{theorem}\label{th66}
Let $F \in {K[z_1,z_2,\ldots,z_n]^{l \times l}}$ and $\det(F)=(z_1-f_1(z_2,\cdots,z_n))\cdots (z_{n-1}-f_{n-1}(z_n))h$, where $h\in K[z_n]$. Then $F$ is equivalent to its Smith form if and only if $J_k(F) = K[{z_1},{z_2}, \ldots ,{z_n}]$ for $k=1,\ldots,l$.
\end{theorem}
\begin{proof}
Let $\gamma = (z_1-f_1(z_2,\ldots,z_n))\cdots (z_{n-1}-f_{n-1}(z_n))$. Assume that the Smith form of $F$ is
\[S = diag\{h_1,h_2,\ldots, \gamma h_l\},\]
where $h_1,h_2,\ldots,h_l\in K[z_n]$ satisfy $h_1\mid h_2\mid \cdots \mid h_l$ and $h_1\cdots h_l=h$.

Necessity: If $F \sim S$, by Lemma \ref{le2}, $J_k(F) =J_k(S) = K[{z_1},{z_2}, \ldots ,{z_n}]$, where $k = 1, \ldots , l$.

Sufficiency:
By using a method similar to that in Theorem \ref{th6}, we need only replace the isomorphism $\theta$ with the one defined by  $\theta(z_1)=z_1-f(z_2,\ldots,z_n),\ldots, \theta(z_{n-1})=z_{n-1}-f_{n-1}(z_n)$ and $\theta(z_n)= z_n$. 
We can prove that $F \sim diag\{h_1,\ldots,h_l\}\cdot G_1$, where $\det(G_1)=\gamma$, $d_{l-1}(G_1) = 1$ and $J_{l-1}(G_1) =K[{z_1},{z_2}, \ldots ,{z_n}]$. By Lemma \ref{le6}, there exists a unimodular matrix $W_1 \in K[{z^c_{n-1}}]^{l\times l}$ such that 
\[
W_1G_1 =  \left( {\begin{array}{*{20}{c}}
{I_{l-1}}&{}\\
{}& z_{n-1}-f_{n-1}(z_n)
\end{array}} \right)\cdot G_2.
\]
Hence, we have the following equivalence relation:
$$F \sim diag\{h_1,h_2,\ldots,h_l\}\cdot W_1^{-1}\cdot \left( {\begin{array}{*{20}{c}}
{I_{l-1}}&{}\\
{}& z_{n-1}-f_{n-1}(z_n)
\end{array}} \right) G_2,
$$
where $G_2 \in K[z_1,z_2, \ldots ,{z_n}]^{l\times l}$. 
Let 
$$A_1 = diag\{h_1,h_2,\ldots,h_l\}\cdot W_1^{-1}\cdot \left( {\begin{array}{*{20}{c}}
{I_{l-1}}&{}\\
{}& z_{n-1}-f_{n-1}(z_n)
\end{array}} \right).$$
Using the same method as in Theorem \ref{th1}, there exist two unimodular matrices $U_2,V_2 \in K[{z_1},{z_2}, \ldots ,{z_n}]^{l\times l}$ such that 
$A_1 = U_2\cdot diag\{h_1,h_2,\ldots, (z_{n-1}-f_{n-1}(z_n))h_l\}
\cdot V_2.$
Therefore, $F \sim diag\{h_1,h_2,\ldots, (z_{n-1}-f_{n-1}(z_n))h_l\}\cdot V_2G_2$. Let $G_3 = V_2G_2$. Then we have $d_{l-1}(G_3)= 1$, $J_{l-1}(G_3)=K[{z_1},{z_2}, \ldots ,{z_n}]$. It is clear that $(z_{n-2}-f_{n-2}(z_{n-1},z_n))\mid \det(G_3)$. By Lemma \ref{le6}, there exists a unimodular matrix $W_2 \in K[{z^c_{n-2}}]^{l\times l}$ such that 
$
W_2G_3 =  diag\{1,\ldots,1,z_{n-2}-f_{n-2}(z_{n-1},z_n)\}\cdot G'_4.$ By Theorem \ref{th4} again, we have 
$$F \sim diag\{h_1,h_2,\ldots, (z_{n-2}-f_{n-2}(z_{n-1},z_n))(z_{n-1}-f_{n-1}(z_n))h_l\} \cdot G_4.$$
where $G_4 \in K[z_1,z_2, \ldots ,{z_n}]^{l\times l}$. By repeating the above procedure $n-3$ times, we obtain 
$$F \sim diag\{h_1,h_2,\ldots, \gamma h_l\}\cdot G,$$
where $G \in K[z_1,z_2, \ldots ,{z_n}]^{l\times l}$. It is straightforward that $G$ is unimodular, hence $F \sim S$.
\end{proof}

The above theorem resolves Problem \ref{pro2}. Subsequently, we extend this result  to non-square and rank-deficient matrices.

\begin{theorem}\label{th2.2}
Let $F \in {K[{z_1},{z_2}, \ldots ,{z_n}]^{l \times m}}(l<m)$ with rank $r (1\le r \le l)$, $d_r(F) = (z_1-f_1(z_2,\ldots,z_n))\cdots (z_{n-1}-f_{n-1}(z_n))h$, where $h\in K[z_n]$. Then $F$ is equivalent to its Smith form if and only if $J_k(F) = K[{z_1},{z_2}, \ldots ,{z_n}]$ for $k = 1, \ldots , r$.
\end{theorem}
\begin{proof}
According to  Lemma \ref{le00} and Theorem \ref{th66}, the proof is similar to that of Theorem \ref{th1.2} and is therefore omitted.
\end{proof}


\begin{remark}
According to the proofs of Theorems \ref{th4} and \ref{th66}, the product preceding \( h \in K[z_n] \) in Theorem~\ref{th2.2} is not required to contain all \( n-1 \) factors of the form \( z_i - f_i(z_{i+1},\ldots,z_n) \), where $i = 1,\ldots,n-1$. Specifically, it may contain only a subset of such factors—for instance, it could take the form $(z_{n-1} - f_{n-1}(z_n))\,h$, or $(z_1 - f_1(z_2, \ldots, z_n))(z_{n-1} - f_{n-1}(z_n))\,h$. 
\end{remark}

We now illustrate the validity and effectiveness of Theorem \ref{th66} through a specific example.
\begin{example}\label{exa1}
Consider the matrix $F\in \mathbb{Q}[x,y,z]^{3\times 3}$ given by
\[
F=
\begin{pmatrix}
F_{11} & F_{12} & F_{13}\\
F_{21} & F_{22} & F_{23}\\
F_{31} & F_{32} & F_{33}
\end{pmatrix},
\]
where $ \mathbb{Q}$ is the rational number field and
\begin{align*}
F_{11} &= y z^{4}+z^{5}+y^{2} z^{2}+y z^{3}-y z^{2}-z^{3}-y^{2}-y z,\\
F_{12} &= -x y z^{4}-x z^{5}+y z^{5}-x y^{2} z^{2}-x y z^{3}-x z^{4}+y^{2} z^{3}+x z^{3}-y z^{3}\\*
        &\quad +x y^{2}+x y z+x z^{2}-y^{2} z+x y,\\
F_{13} &= y z^{6}-x z^{5}+y^{2} z^{4}-x y z^{3}-y z^{4}+x z^{3}-y^{2} z^{2}+x y z+z^{3}+x z+x-z,\\[4pt]
F_{21} &= -x y z^{4}-x z^{5}-x y^{2} z^{2}-x y z^{3}+x y z^{2}+x z^{3}+x y^{2}+x y z-z-1,\\
F_{22} &= x^{2} y z^{4}+x^{2} z^{5}-x y z^{5}+x^{2} y^{2} z^{2}+x^{2} y z^{3}+x^{2} z^{4}-x y^{2} z^{3}-x^{2} z^{3}\\*
        &\quad +x y z^{3}-x^{2} y^{2}-x^{2} y z-x^{2} z^{2}+x y^{2} z-x^{2} y+x z+x,\\
F_{23} &= -x y z^{6}+x^{2} z^{5}-x y^{2} z^{4}+x^{2} y z^{3}+x y z^{4}-x^{2} z^{3}+x y^{2} z^{2}-x^{2} y z\\*
        &\quad -x z^{3}-x^{2} z-x^{2}-2 y z-x-2 y,\\[4pt]
F_{31} &= -y^{2} z^{4}-2 y z^{5}-z^{6}-y^{3} z^{2}-2 y^{2} z^{3}-y z^{4}+y^{2} z^{2}+2 y z^{3}+z^{4}\\*
        &\quad +y^{3}+2 y^{2} z+y z^{2}-z-1,\\
F_{32} &= x y^{2} z^{4}+2 x y z^{5}+x z^{6}-y^{2} z^{5}-y z^{6}+x y^{3} z^{2}+2 x y^{2} z^{3}+2 x y z^{4}\\*
        &\quad +x z^{5}-y^{3} z^{3}-y^{2} z^{4}-x y z^{3}-x z^{4}+y^{2} z^{3}+y z^{4}-x y^{3}-2 x y^{2} z\\*
        &\quad -2 x y z^{2}-x z^{3}+y^{3} z+y^{2} z^{2}-x y^{2}-x y z+x z+x,\\
F_{33} &= -y^{2} z^{6}-y z^{7}+x y z^{5}+x z^{6}-y^{3} z^{4}-y^{2} z^{5}+x y^{2} z^{3}+x y z^{4}+y^{2} z^{4}\\*
        &\quad +y z^{5}-x y z^{3}-x z^{4}+y^{3} z^{2}+y^{2} z^{3}-x y^{2} z-x y z^{2}-y z^{3}-z^{4}\\*
        &\quad -x y z-x z^{2}-x y-2 x z-y z+z^{2}-x-2 y-1.
\end{align*}

By calculation, we have $d_1(F)=1$, $d_2(F) = z+1$, and $\det(F) = (x-yz)(y+z^2)(z^3 + z^2 - z - 1)$. Let $h = z^3 + z^2 - z - 1$; we can decompose $h$ as $(z+1)^2(z-1)$. Then the Smith form of $F$ is
$$
S = \left( {\begin{array}{*{20}{c}}
 1 & 0 & 0 \\
0 & z+1 & 0\\
 0 & 0 & (x-yz)(y+z^2)(z^2-1)
\end{array}} \right).
$$

Using {\sc Maple}, we can compute all the first- and second-order reduced minors of $F$ and find that the reduced Gr\"{o}bner basis of the ideal generated by these minors is $\{1\}$. This implies that $J_1(F) = J_2(F) = \mathbb{Q}[x,y,z]$. In addition, it is obvious that $J_3(F) = \mathbb{Q}[x,y,z]$.
Therefore, we can apply Theorem \ref{th66} to reduce $F$ to its equivalent Smith form.

Since $(z+1) \mid d_2(F)$ and $J_1(F)= \mathbb{Q}[x,y,z]$, by Lemma \ref{le5}, there exists a unimodular matrix $U_1=\left( {\begin{array}{*{20}{c}}
y +z -x  & -1 & 1 \\
 x  & 1 & 0 \\
 -1 & 0 & 0 
\end{array}} \right)$ such that $U_1F = \left( {\begin{array}{*{20}{c}}
 1 & 0 & 0 \\
0 & z+1 & 0\\
 0 & 0 & z+1
\end{array}} \right)G_1$. By Lemma \ref{le4}, we have $d_2(G_1) = 1$, $J_2(G_1)= \mathbb{Q}[x,y,z]$. Since $(z-1)\mid d_3(G_1)$. By Lemma \ref{le5} again, there exists a unimodular matrix $U_2=\left( {\begin{array}{*{20}{c}}
1 & 0  & 0  \\
 -y & 1 & 0 \\
-x & 0 & 1 
\end{array}} \right)$ such that $U_2G_1 = \left( {\begin{array}{*{20}{c}}
 1 & 0 & 0 \\
0 & 1 & 0\\
 0 & 0 & z-1
\end{array}} \right)G_2$, where 
$$
G_2 = \left[\begin{array}{ccc}
0 & 0 & -1 \\
 -1 & x  & -x -y  \\
 -y z^{2}-z^{3}-y^{2}-y z  & G_2[3,2]  & G_2[3,3]
\end{array}\right],
$$
where $G_2[3,2]=x yz^{2}+xz^{3}-yz^{3}+x y^{2}+x y z +x z^{2}-y^{2} z +x y$, $G_2[3,3]=-y z^{4}+x z^{3}-y^{2} z^{2}+x y z -z $. 

Let 
$$A  = \left( {\begin{array}{*{20}{c}}
 1 & 0 & 0 \\
0 & z+1 & 0\\
 0 & 0 & z+1
\end{array}} \right)U^{-1}_2\left( {\begin{array}{*{20}{c}}
 1 & 0 & 0 \\
0 & 1 & 0\\
 0 & 0 & z-1
\end{array}} \right).$$
It is easy to verify that there exists a unimodular matrix $U'= \left( {\begin{array}{*{20}{c}}
 1 & 0 & 0 \\
-yz-1 & 1 & 0\\
 -xz-1 & 0 & 1
\end{array}} \right)$ such that 
$$
U'A = \left( {\begin{array}{*{20}{c}}
 1 & 0 & 0 \\
0 & z+1 & 0\\
 0 & 0 & (z+1)(z-1)
\end{array}} \right).
$$
Thus, we have
$$
F = U^{-1}_1U'^{-1}\left( {\begin{array}{*{20}{c}}
 1 & 0 & 0 \\
0 & z+1 & 0\\
 0 & 0 & (z+1)(z-1)
\end{array}} \right)G_2.
$$
Since $d_2(G_2) = 1$, $J_2(G_2)= \mathbb{Q}[x,y,z]$ and $(y+z^2)\mid d_3(G_2)$. By Lemma \ref{le6}, there exists a unimodular matrix $U_3=\left( {\begin{array}{*{20}{c}}
1 & 0  & 0  \\
 -x & 1 & 0 \\
 -z  & 0  & 1 
\end{array}} \right)$ such that $U_3G_2 = \left( {\begin{array}{*{20}{c}}
 1 & 0 & 0 \\
0 & 1 & 0\\
 0 & 0 & y+z^2
\end{array}} \right)G_3$, where 
$$
G_3 = \left( {\begin{array}{*{20}{c}}
0 & 0 & -1 \\
 -1 & x  & -y  \\
 -y -z  & \left(y +z \right) x -y z +x  & \left(-y z +x \right) z 
\end{array}} \right).
$$
At this point, we have $d_2(G_3) = 1$, $J_2(G_3)= \mathbb{Q}[x,y,z]$ and $(x-yz)\mid d_3(G_3)$. By Lemma \ref{le6} again, there exists a unimodular matrix $U_4=\left( {\begin{array}{*{20}{c}}
1 & 0 & 0 
\\
 -y  & 1 & 0 
\\
 y \left(y +z \right) & -y -z  & 1 
\end{array}} \right)$ such that $U_4G_3 = \left( {\begin{array}{*{20}{c}}
 1 & 0 & 0 \\
0 & 1 & 0\\
 0 & 0 & x-yz
\end{array}} \right)G_4$, where 
$
G_4 = \left( {\begin{array}{*{20}{c}}
0 & 0 & -1 
\\
 -1 & x  & 0 
\\
 0 & 1 & z  
\end{array}} \right).
$
Combining the above equalities, $F$ can be reexpressed as:
$$
F = U^{-1}_1U'^{-1}\left( {\begin{array}{*{20}{c}}
 1 & 0 & 0 \\
0 & z+1 & 0\\
 0 & 0 & z^2-1
\end{array}} \right)U^{-1}_3\left( {\begin{array}{*{20}{c}}
 1 & 0 & 0 \\
0 & 1 & 0\\
 0 & 0 & y+z^2
\end{array}} \right)U^{-1}_4\left( {\begin{array}{*{20}{c}}
 1 & 0 & 0 \\
0 & 1 & 0\\
 0 & 0 & x-yz
\end{array}} \right)G_4.
$$

By Theorem \ref{th1}, there exists a unimodular matrix $P=\left( {\begin{array}{*{20}{c}}
 1 & 0 & 0 \\
xz+1 & 1 & 0\\
 z^3-z & 0 & 1
\end{array}} \right)$ such that
$$
\left( {\begin{array}{*{20}{c}}
 1 & 0 & 0 \\
0 & z+1 & 0\\
 0 & 0 & z^2-1
\end{array}} \right)U^{-1}_3\left( {\begin{array}{*{20}{c}}
 1 & 0 & 0 \\
0 & 1 & 0\\
 0 & 0 & y+z^2
\end{array}} \right) = P\left( {\begin{array}{*{20}{c}}
 1 & 0 & 0 \\
0 & z+1 & 0\\
 0 & 0 & (z^2-1)(y+z^2)
\end{array}} \right).
$$
By Theorem \ref{th4}, we further obtain that there exists a unimodular matrix  
$$Q = \left( {\begin{array}{*{20}{c}}
 1 & 0 & 0 \\
yz+1 & z+1 & 0\\
0 & (y+z)(z^2-1)(y+z^2) & (y+z)(z^2-1)(y+z^2)(x-yz)
\end{array}} \right)$$ 
such that
$$
\left( {\begin{array}{*{20}{c}}
 1 & 0 & 0 \\
0 & z+1 & 0\\
 0 & 0 & (z^2-1)(y+z^2)
\end{array}} \right)U^{-1}_4\left( {\begin{array}{*{20}{c}}
 1 & 0 & 0 \\
0 & 1 & 0\\
 0 & 0 & x-yz
\end{array}} \right) = Q\left( {\begin{array}{*{20}{c}}
 1 & 0 & 0 \\
0 & z+1 & 0\\
 0 & 0 & (z^2-1)(y+z^2)(x-yz)
\end{array}} \right).
$$

Therefore, 
$$
F = U^{-1}_1U'^{-1}PQ\left( {\begin{array}{*{20}{c}}
 1 & 0 & 0 \\
0 & z+1 & 0\\
 0 & 0 & (z^2-1)(y+z^2)(x-yz)
\end{array}} \right)G_4.
$$
Since $U^{-1}_1U'^{-1}PQ$ and $G_4$ are unimodular, it follows that $F$ is equivalent to $S$.
\end{example}

\section{The equivalence of matrices via algebra isomorphisms}
In this section, we employ algebra isomorphisms to deal with the equivalence problem of a new class of polynomial matrices, which is precisely Problem \ref{pro3}. Let $F \in {K[z_1,z_2,\ldots,z_n]^{l \times l}}$  with $\det(F)=g_1g_2\cdots g_{n-1}h$, where $h\in K[z_n]$ and for $i=1,\cdots,n-1$, $g_i = \sum_{j=1}^{n-1} a_{ij}z_j+b_i$ with $a_{ij},b_i\in K$. Moreover, the $n-1$ linear polynomials $g_1, \ldots, g_{n-1}$ are linearly independent over $K$.  By constructing an automorphism of $K[z_1,z_2,\ldots,z_n]$, we prove that $F$ is equivalent to its Smith form. 

\begin{definition}\label{def41}
Let \( g_i \in K[z_1,\ldots,z_{n-1}] \) be $n-1$ linear polynomials of the form  
$
g_i = a_{i1}z_1 + a_{i2}z_2 + \ldots + a_{i,n-1}z_{n-1} + b_i,
$  
where \( a_{ij}, b_i \in K \) for \( i = 1,\ldots,n-1 \). Define an endomorphism $\phi$ as follows:
\[
\begin{aligned}
\phi : K[z_1,z_2,\ldots,z_n&]\to K[z_1,z_2,\ldots,z_n] \\
z_1 &\mapsto g_1, \\
z_2 &\mapsto g_2, \\
&\cdots \\
z_{n-1}&\mapsto g_{n-1}, \\
z_n &\mapsto z_n.
\end{aligned}
\]
\end{definition}

\begin{lemma}\label{le88}
Let $\phi, g_1, \ldots, g_{n-1}$ be given as in Definition \ref{def41}. If $g_1, \ldots, g_{n-1}$ are linearly independent over $K$. Then $\phi$ is an automorphism.
\end{lemma}
\begin{proof}
Let
$$
A = \left( {\begin{array}{*{20}{c}}
{a_{11}}&{a_{12}}&{\cdots}&{a_{1,n-1}}&0&b_1\\
{a_{21}}&{a_{22}}&{\cdots}&{a_{2,n-1}}&0&b_2\\
{\vdots}&{\vdots}& \ddots &{\vdots}&\vdots&\vdots\\
{a_{n-1,1}}&{a_{n-1,2}}&{\cdots}&{a_{n-1,n-1}}&0&b_{n-1}\\
{0}&{0}&{\cdots}&{0}&1&0\\
{0}&{0}&{\cdots}&{0}&0&1
\end{array}} \right) \in K^{(n+1)\times (n+1)}.
$$
Since $g_1, \ldots, g_{n-1}$ are linearly independent, $A$ is invertible. In addition, we have 
$$
\phi \left( {\begin{array}{*{20}{c}}
{z_1}\\
{z_2}\\
{\vdots}\\
{z_{n-1}}\\
{z_n}\\
{1}
\end{array}} \right) = \left( {\begin{array}{*{20}{c}}
\phi({z_1})\\
\phi({z_2})\\
{\vdots}\\
\phi({z_{n-1}})\\
\phi({z_n})\\
\phi({1})
\end{array}} \right) = A \left( {\begin{array}{*{20}{c}}
{z_1}\\
{z_2}\\
{\vdots}\\
{z_{n-1}}\\
{z_n}\\
{1}
\end{array}} \right).
$$
Consider another endomorphism $\psi$ of $K[z_1,z_2,\ldots,z_n]$, defined as follows:
$$
\psi \left( {\begin{array}{*{20}{c}}
{z_1}\\
{z_2}\\
{\vdots}\\
{z_{n-1}}\\
{z_n}\\
{1}
\end{array}} \right) = A^{-1}\left( {\begin{array}{*{20}{c}}
{z_1}\\
{z_2}\\
{\vdots}\\
{z_{n-1}}\\
{z_n}\\
{1}
\end{array}} \right).
$$
Then
$$
\psi\cdot \phi\left( {\begin{array}{*{20}{c}}
{z_1}\\
{z_2}\\
{\vdots}\\
{z_{n-1}}\\
{z_n}\\
{1}
\end{array}} \right) = \psi\left( {\begin{array}{*{20}{c}}
A \left( {\begin{array}{*{20}{c}}
{z_1}\\
{z_2}\\
{\vdots}\\
{z_{n-1}}\\
{z_n}\\
{1}
\end{array}} \right)
\end{array}} \right)= A\cdot \psi \left( {\begin{array}{*{20}{c}}
{z_1}\\
{z_2}\\
{\vdots}\\
{z_{n-1}}\\
{z_n}\\
{1}
\end{array}} \right)=A\cdot A^{-1}\left( {\begin{array}{*{20}{c}}
{z_1}\\
{z_2}\\
{\vdots}\\
{z_{n-1}}\\
{z_n}\\
{1}
\end{array}} \right)  = \left( {\begin{array}{*{20}{c}}
{z_1}\\
{z_2}\\
{\vdots}\\
{z_{n-1}}\\
{z_n}\\
{1}
\end{array}} \right).
$$
On the other hand,
$$
\phi\cdot \psi\left( {\begin{array}{*{20}{c}}
{z_1}\\
{z_2}\\
{\vdots}\\
{z_{n-1}}\\
{z_n}\\
{1}
\end{array}} \right) = \phi\left( {\begin{array}{*{20}{c}}
A^{-1} \left( {\begin{array}{*{20}{c}}
{z_1}\\
{z_2}\\
{\vdots}\\
{z_{n-1}}\\
{z_n}\\
{1}
\end{array}} \right)
\end{array}} \right)= A^{-1}\cdot \phi \left( {\begin{array}{*{20}{c}}
{z_1}\\
{z_2}\\
{\vdots}\\
{z_{n-1}}\\
{z_n}\\
{1}
\end{array}} \right)=A^{-1}\cdot A\left( {\begin{array}{*{20}{c}}
{z_1}\\
{z_2}\\
{\vdots}\\
{z_{n-1}}\\
{z_n}\\
{1}
\end{array}} \right)  = \left( {\begin{array}{*{20}{c}}
{z_1}\\
{z_2}\\
{\vdots}\\
{z_{n-1}}\\
{z_n}\\
{1}
\end{array}} \right).
$$
It follows that $\phi\cdot \psi = \psi\cdot \phi = {\bf 1}$, where ${\bf 1}$ is the unit of $K[z_1,z_2,\ldots,z_n]$. Therefore, $\phi$ is an automorphism.
\end{proof}

\begin{lemma}\label{le888}
Let $F \in {K[{z_1},{z_2}, \ldots ,{z_n}]^{l \times l}}$ with rank $r (1\le r \le l)$, $\phi$ be given as in Definition \ref{def41}. If $J_k(F) = K[{z_1},{z_2}, \ldots ,{z_n}]$ for $k=1,\ldots,r$, then $J_k(\phi(F))= K[{z_1},{z_2}, \ldots ,{z_n}]$ and $d_k(\phi(F))=\phi(d_k(F))$ for $k=1,\ldots,r$.
\end{lemma}
\begin{proof}
Let $u_1,\ldots, u_\beta$ be all the $k_0 \times k_0$ minors of $F$, and $v_1,\ldots, v_\beta$ be the corresponding reduced minors, where $1\le k_0 \le r$. Then $u_i = d_{k_0}(F)v_i$ for $i=1,\ldots,\beta$. Since $\phi$ is an automorphism, the images $\phi(u_1),\ldots,\phi(u_\beta)$ are the $k_0 \times k_0$ minors of $\phi(F)$ and $\phi(u_i) = \phi(d_{k_0}(F))\phi(v_i)$ for $i=1,\cdots,\beta$. Furthermore, the fact that $J_{k_0}(F)= K[{z_1},{z_2}, \ldots ,{z_n}]$ implied that there exist $b_1,\ldots,b_\beta \in K[{z_1},{z_2}, \ldots ,{z_n}]$ such that 
$$
b_1v_1+ b_2v_2+\ldots + b_\beta v_\beta = 1.
$$
Applying $\phi$ to the equation above yields
$$
\phi(b_1)\phi(v_1)+ \phi(b_2)\phi(v_2)+\ldots + \phi(b_\beta) \phi(v_\beta) = 1.
$$
Therefore, we have $gcd(\phi(v_1),\phi(v_2),\ldots,\phi(v_\beta))=1$. Combined with $\phi(u_i) = \phi(d_{k_0}(F)) \phi(v_i)$, we further have
$\phi(v_1),\phi(v_2),\ldots,\phi(v_\beta)$ are the $k_0 \times k_0$ reduced minors of $\phi(F)$. 
It follows that $J_{k_0}(\phi(F))= K[{z_1},{z_2}, \ldots ,{z_n}]$ and $d_{k_0}(\phi(F))=\phi(d_{k_0}(F))$, where $1\le k_0 \le r$.
\end{proof}


Let $g_1, \ldots, g_{n-1}\in K[z_1,z_2,\ldots,z_{n-1}]$ be $n-1$ linear polynomials defined by 
$$g_i = \sum_{j=1}^{n-1} a_{ij} z_j + b_i,$$
where $a_{ij}, b_i \in K$, $i=1,\ldots,n-1$. The main results of this section are as follows.

\begin{theorem}\label{th77}
Let $F \in {K[z_1,z_2,\ldots,z_n]^{l \times l}}$ and $\det(F)=g_1g_2\cdots g_{n-1}h$, where $h\in K[z_n]$ and $g_1,g_2,\ldots, g_{n-1}$ are linearly independent. Then $F$ is equivalent to its Smith form if and only if $J_k(F) = K[{z_1},{z_2}, \ldots ,{z_n}]$ for $k=1,\ldots,l$.
\end{theorem}
\begin{proof}
Necessity: Assume that $S$ is the Smith form of $F$. It is straightforward to check that $J_k(S) = K[{z_1},{z_2}, \ldots ,{z_n}]$ for $k=1,\ldots,l$. If $F$ is equivalent to $S$, by Lemma \ref{le2}, we have that $J_k(F) = J_k(S) =K[{z_1},{z_2}, \ldots ,{z_n}]$ for $k=1,\ldots,l$.

Sufficiency: Let $\phi$ be given as in Definition \ref{def41}. 
By Lemma \ref{le88}, $\phi$ is an automorphism. Thus, $\phi^{-1}$ is also an automorphism of $K[{z_1},{z_2}, \ldots ,{z_n}]$ and 
$$\phi^{-1}(g_1,g_2,\ldots,g_{n-1},z_n) = (z_1,z_2,\ldots,z_{n-1},z_n).$$
Write \(F = (f_{ij})_{l \times l}\) and set \(\phi^{-1}(F) = (\phi^{-1}(f_{ij}))_{l \times l}\). Since \(J_k(F) = K[{z_1},{z_2}, \ldots ,{z_n}]\), by Lemma \ref{le888}, we have $d_k(\phi^{-1}(F))=\phi^{-1}(d_k(F))$ and \(J_k(\phi^{-1}(F)) = K[{z_1},{z_2}, \ldots ,{z_n}]\) for \(k = 1, 2, \ldots, l\). Thus, it follows from \(\det(F) = g_1g_2\cdots g_{n-1}h\) that 
$$\det(\phi^{-1}(F)) = \phi^{-1}(\det(F)) = z_1z_2\cdots z_{n-1}\phi^{-1}(h).$$
Note that $h\in K[z_n]$ and the variable $z_n$ is unchanged by $\phi^{-1}$. Hence $\phi^{-1}(h) = h$. By Theorem \ref{th66}, $\phi^{-1}(F)$ is equivalent to its Smith form; i.e., there exist two unimodular matrices \(U, V \in {K[{z_1},{z_2}, \ldots ,{z_n}]^{l \times l}}\) such that \(U\cdot \phi^{-1}(F)\cdot V = S\), where \(S\) is the Smith form of \(\phi^{-1}(F)\). Let $S = diag\{\Phi_1,\Phi_2,\ldots,\Phi_l\}$, where $\Phi_1,\Phi_2,\ldots,\Phi_l \in K[{z_1},{z_2}, \ldots ,{z_n}]$ and satisfies the following relation:
\begin{equation}\label{eq41}
\Phi_k = \frac{d_k(\phi^{-1}(F))}{d_{k-1}(\phi^{-1}(F))},\ k = 1, 2, \ldots, l.
\end{equation}
We claim that \(\phi(S)\) is the Smith form of \(F\). From equation (\ref{eq41}), we have $d_k(\phi^{-1}(F)) = \Phi_k\cdot d_{k-1}(\phi^{-1}(F))$. Hence, 
\begin{equation}\label{eq42}
\phi^{-1}(d_k(F)) = \Phi_k\cdot \phi^{-1}(d_{k-1}(F)).\end{equation}
By multiplying both sides of equation (\ref{eq42}) by \(\phi\), we obtain
\begin{equation}\label{eq43}
\phi(\Phi_k) = \frac{d_k((F))}{d_{k-1}((F))},\ k = 1, 2, \ldots, l.
\nonumber
\end{equation}
Therefore, $\phi(S) = diag\{\phi(\Phi_1),\phi(\Phi_2),\ldots,\phi(\Phi_l)\}$ is the Smith form of \(F\). From \(U\cdot \phi^{-1}(F)\cdot V = S\), we have that
\(\phi(U) \cdot F \cdot \phi(V) = \phi(S)\), by Lemma \ref{le888}, $\phi(U), \phi(V)\in K[{z_1},{z_2}, \ldots ,{z_n}]^{l \times l}$ are unimodular matrices. Therefore, we conclude that $F$ is equivalent to its Smith form.
\end{proof}

\begin{theorem}\label{th2.23}
Let $F \in {K[{z_1},{z_2}, \ldots ,{z_n}]^{l \times m}}(l<m)$ with rank $r (1\le r \le l)$ and $d_r(F) = g_1g_2\cdots g_{n-1}h$, where $h\in K[z_n]$ and $g_1,g_2,\ldots, g_{n-1}$ are linearly independent. Then $F$ is equivalent to its Smith form if and only if $J_k(F) = K[{z_1},{z_2}, \ldots ,{z_n}]$ for $k = 1, \ldots , r$.
\end{theorem}
\begin{proof}
The proof is similar to that of Theorem \ref{th1.2} and is therefore omitted.
\end{proof}

\begin{remark}
It is worth noting that the number of linearly independent linear polynomials in Theorem \ref{th2.23} can be less than $n-1$. In such cases, we only need to embed the coefficient matrix of these linear polynomials into a $(n-1)\times (n-1)$ invertible matrix. This embedding preserves the validity of Lemma \ref{le88}.
\end{remark}


In what follows,  we use an example to illustrate the effectiveness of Theorem \ref{th77}. We shall omit the calculation steps that require the application of Theorem \ref{th66}, since the relevant process has already been illustrated in Example \ref{exa1}.
\begin{example}\label{exa2}
Consider the matrix $F\in \mathbb{Q}[x,y,z]^{3\times 3}$ given by
\[
F=
\begin{pmatrix}
F_{11} & F_{12} & F_{13}\\
F_{21} & F_{22} & F_{23}\\
F_{31} & F_{32} & F_{33}
\end{pmatrix},
\]
where $ \mathbb{Q}$ is the rational number field and
\begin{align*}
F_{11} &= z+1,\quad
F_{12} = x+y,\quad
F_{13}= x^{3}+x^{2}y-xy^{2}-y^{3}-1,\\[4pt]
F_{21} &= x^{3}z -x^{2}yz -x^{2}z^{2}-xy^{2}z +y^{3}z +y^{2}z^{2}+x^{3}-x^{2}y -2x^{2}z -xy^{2}\\ \*
&\quad +xz^{2}+y^{3}+2y^{2}z -yz^{2}-x^{2}+2xz +y^{2}-2yz +z^{2}+x -y +z,\\
F_{22} &= -x^{3}z -x^{2}yz +xy^{2}z +y^{3}z -x^{3}-x^{2}y +x^{2}z +xy^{2}+y^{3}-y^{2}z\\ \*
&\quad +x^{2}-y^{2}+z+1,\\
F_{23} &= -x^{5}z -x^{4}yz +2x^{3}y^{2}z +2x^{2}y^{3}z -xy^{4}z -y^{5}z -x^{5}-x^{4}y +x^{4}z\\ \*
&\quad +2x^{3}y^{2}+2x^{2}y^{3}-2x^{2}y^{2}z -xy^{4}-y^{5}+y^{4}z +x^{4}-2x^{2}y^{2}+y^{4}\\ \*
&\quad +2x^{2}z -2y^{2}z +2x^{2}-xz -2y^{2}+yz -x +y,\\[4pt]
F_{31} &= x^{2}z -xz^{2}-y^{2}z -yz^{2}+x^{2}-3xz -y^{2}-3yz -2x -2y,\\
F_{32} &= -x^{2}z -2xyz -y^{2}z -2x^{2}-4xy -2y^{2},\\
F_{33} &= -x^{4}z -2x^{3}yz +2xy^{3}z +y^{4}z -2x^{4}-4x^{3}y +4xy^{3}+2y^{4}\\ \*&\quad +xz +yz +2x +2y.
\end{align*} 

Similar to Example \ref{exa1}, we first calculation that $d_1(F)=1$, $d_2(F) = z+1$, and $\det(F) = (z + 1)^2(x - y)(x + y)$. Let $g_1 = x-y$ and $g_2 = x+y$ be two linear polynomials. Then the Smith form of $F$ is
$$
S = \left( {\begin{array}{*{20}{c}}
 1 & 0 & 0 \\
0 & z+1 & 0\\
 0 & 0 & g_1g_2(z+1)
\end{array}} \right).
$$

Using {\sc Maple} to compute the Gr\"{o}bner bases of the ideals $J_i(F)$, we further obtain that $J_i(F) = \mathbb{Q}[x,y,z]$, where $i=1,2,3$. In addition, $g_1,g_2$ are linearly independent.
Therefore, by Theorem \ref{th77}, we have $F$ is equivalent to its Smith form $S$.
We construct a homomorphic mapping $\phi$ of $\mathbb{Q}[x,y,z]$ as follows:
\[
\begin{aligned}
\phi :\quad 
x \mapsto g_1,\quad y \mapsto g_2,\quad z \mapsto z.
\end{aligned}
\]

By Lemma \ref{le88}, $\phi$ is an automorphism satisfying \(\phi(x,y,z) = (g_1,g_2,z)\). Also, $\phi^{-1}$ is an automorphism of $\mathbb{Q}[x,y,z]$. Thus, we have
\[
\phi^{-1}(F)=
\begin{pmatrix}
z +1 & y  & xy^{2}-1 \\
F_{21}  & F_{22} & F_{23}\\
 x y z -yz^{2}+x y -3 y z -2 y & -y^{2} z -2 y^{2}& -xy^{3} z -2 xy^{3}+yz +2y 
\end{pmatrix},
\]
where $F_{21} = x^{2} y z -x y z^{2}+x^{2} y -2 x y z +x z^{2}-x y +2 x z +z^{2}+x +z$, $F_{22} = -x y^{2} z -x y^{2}+x y z +x y +z +1$, $F_{23} = -x^{2} y^{3} z -x^{2} y^{3}+x^{2} y^{2} z +x^{2} y^{2}+2 x y z +2 x y -x z -x $.

By Lemma \ref{le888} (verified independently with {\sc Maple}), we have $d_1(\phi^{-1}(F))=1$, $d_2(\phi^{-1}(F)) = z+1$, and $\det(\phi^{-1}(F)) = xy(z + 1)^2$. Following the calculation procedure in Theorem \ref{th66}, we can find two unimodular matrices 
$$
M= \begin{pmatrix}
1 & 0 & 0 \\
 -x y z -x y +x z +x  & 1 & x  \\
 -y z -2 y  & 0 & 1  
\end{pmatrix},  \quad  N = \begin{pmatrix}
z +1 & y  & x \,y^{2}-1 \\
 z  & 1 & x y  \\
 1 & 0 & 0 
\end{pmatrix}
$$
such that $\phi^{-1}(F) = M\phi^{-1}(S)N$, where $\phi^{-1}(S)$ is the Smith form of $\phi^{-1}(F)$. Consequently, we obtain two unimodular matrices $\phi(M), \phi(N)\in \mathbb{Q}[x,y,z]^{3\times 3}$ satisfying $F = \phi(M)\cdot S\cdot \phi(N)$.
\end{example}

\section{Conclusions}
In this paper, we have investigated the equivalence problem for three classes of multivariate polynomial matrices to their Smith forms. 
Let $F \in K[{z_1},{z_2}, \ldots ,{z_n}]^{l \times l}$ and $h\in K[z_n]$. We established a criterion under which a matrix $F$ with $\det(F)=(z_1-f(z_2,\ldots,z_n))h$  is equivalent to its Smith form.
This criterion was subsequently extended to handle matrices whose determinant is of the form $(z_1-f_1(z_2,\ldots,z_n))\cdots (z_{n-1}-f_{n-1}(z_n))h$.
By leveraging the ring isomorphism $\phi$ that we constructed, where \(\phi: K[z_1,z_2,\ldots,z_n] \to K[z_1,z_2,\ldots,z_n]\) is defined by \(\phi(z_i)=g_i\) for \(i=1,\ldots,n-1\) and \(\phi(z_n)=z_n\), we demonstrated that any matrix $F$ with $\det(F)=g_1g_2\cdots g_{n-1}h$ is equivalent to its Smith form, where $g_1,g_2,\ldots, g_{n-1}$ are linearly independent linear polynomials.
The non-square and rank-deficient cases are also handled.
Notably, these criteria can be effectively verified by computing the reduced Gr\"{o}bner bases of the polynomial ideals generated by the reduced minors of the given matrix. 
The ring isomorphism approach is used to analyze the equivalence of polynomial matrices in this paper. Future work may use this approach to focus on more general classes of multivariate polynomial matrices.



\end{document}